\documentclass{amsart}

\usepackage[all]{xy}

\theoremstyle{plain}
\newtheorem{thm}{Theorem}[section]
\newtheorem{prop}[thm]{Proposition}
\newtheorem{lemma}[thm]{Lemma}

\theoremstyle{definition}
\newtheorem{dfn}[thm]{Definition}

\theoremstyle{remark}
\newtheorem{rem}[thm]{Remark}

\newcommand{\HH}{\mathrm{H}}

\begin{document}

\title{Universal deformation rings for the symmetric group $S_4$}

\author{Frauke M. Bleher}
\address{F.B.: Department of Mathematics\\University of Iowa\\
Iowa City, IA 52242-1419, U.S.A.}
\email{fbleher@math.uiowa.edu}
\thanks{The first author was supported in part by  
NSF Grants DMS01-39737 and DMS06-51332  and NSA Grant
H98230-06-1-0021.}
\author{Giovanna LLosent}
\address{G.L.: Department of Mathematics\\CSU
San Bernardino, CA 92407-2397, U.S.A.}
\email{gllosent@csusb.edu}
\subjclass{Primary 20C20; Secondary 16G20}
\keywords{Universal deformation rings, stable endomorphism rings, 
special biserial algebras}

\begin{abstract}
Let $k$ be an algebraically closed field of characteristic $2$, and let 
$W$ be the ring of infinite Witt vectors over $k$. Let $S_4$ denote the symmetric group on $4$ letters. 
We determine the universal deformation ring $R(S_4,V)$ for every $kS_4$-module $V$ which 
has stable endomorphism ring $k$ and show that $R(S_4,V)$ is isomorphic to either $k$, or
$W[t]/(t^2,2t)$, or the group ring  $W[\mathbb{Z}/2]$. 
This gives a positive answer in this case to a question
raised by the first author and Chinburg whether the universal deformation 
ring of a representation of a  finite group with stable endomorphism ring $k$ 
is always isomorphic to a subquotient ring of the group ring over $W$ of a defect group of the
modular block associated to the representation.
\end{abstract}

\maketitle


\section{Introduction}
\label{s:intro}
\setcounter{equation}{0}
\setcounter{figure}{0}

Let $k$ be an algebraically closed field of characteristic $p>0$ and let $W=W(k)$ be the ring of infinite 
Witt vectors over $k$. Let $G$ be a finite group, and suppose $V$ is a finitely generated $kG$-module.
If the stable endomorphism ring $\underline{\mathrm{End}}_{kG}(V)$ is one-dimensional over $k$,
it was shown in \cite{bc} that $V$ has a universal deformation ring $R(G,V)$. The ring $R(G,V)$ is universal 
with respect to deformations of $V$ over complete local commutative Noetherian rings with residue field $k$
(see \S\ref{s:prelim}). 
In \cite{bc,bl,3sim,diloc}, the isomorphism types of $R(G,V)$ have been determined for $V$ belonging
to cyclic blocks, respectively to various tame blocks with dihedral defect groups
with one or three isomorphism classes of simple modules. 
In the present paper, we will consider the case when $V$ belongs to a particular tame block with two
isomorphism classes of simple modules. 
The key tools used to determine 
the universal deformation rings in all these cases
have been results from modular and ordinary representation theory due to 
Brauer, Erdmann \cite{erd}, 
Linckelmann \cite{linckel,linckel1}, 
Carlson-Th\'{e}venaz \cite{carl2}, 
and others.

The main motivation for studying universal deformation rings for finite groups is that this case helps
understand ring theoretic properties of universal deformation rings for profinite groups $\Gamma$.
The latter have become an important tool in number theory, in particular if $\Gamma$ is a
profinite Galois group
(see e.g. \cite{cornell}, \cite{wiles,taywiles}, \cite{breuil} 
and their references).
In \cite{lendesmit}, de Smit and Lenstra showed
that if $\Gamma$ is an arbitrary profinite group and $V$ is a finite
dimensional vector space over $k$ with a continuous $\Gamma$-action which has a universal
deformation ring  $R(\Gamma,V)$, then $R(\Gamma,V)$ is the inverse limit of the universal 
deformation rings $R(G,V)$ when $G$ runs over all finite discrete quotients of $\Gamma$ through 
which the $\Gamma$-action on $V$ factors. Thus to answer questions about the ring structure of 
$R(\Gamma,V)$, it is natural to first consider the case when $\Gamma=G$ is finite.

Suppose now that $k$ has characteristic $2$ and that $S_4$ denotes the symmetric group on $4$ 
letters. 
In the present paper, we consider the group ring $kS_4$ 
which
is its own (principal) block and 
has two isomorphism classes of simple modules. These are represented by the trivial
simple module $T_0$ and a $2$-dimensional simple module $T_1$ which is inflated from 
the symmetric group $S_3$.
Since the Sylow $2$-subgroups of $S_4$ are dihedral groups of order $8$, $kS_4$ is of
tame representation type.
It was shown in \cite{bc4.9,bc5} that $R(S_4,T_1)\cong W[t]/(t^2,2t)$ which
provides an example of a universal deformation ring which is not a complete intersection. 
The results in \cite{bc4.9,bc5} additionally imply that for every finitely generated $kS_4$-module $V$ 
with stable endomorphism ring $k$, the universal deformation ring $R(S_4,V)$ arises
from arithmetic in the following sense. By \cite[Lemma 3.3 and Thm. 3.7]{bc5}, there are infinitely many real quadratic fields
$L$ such that the Galois group $G_{L,\emptyset}$ of the maximal totally unramified extension of $L$ 
surjects onto $S_4$ and $R(G_{L,\emptyset},V)\cong R(S_4,V)$, where $V$ is viewed as a module for $G_{L,\emptyset}$
via inflation. In particular, our main result, Theorem \ref{thm:supermain},
determines all the universal deformation rings $R(G_{L,\emptyset},V)$.

Our main result is as follows, were $\Omega$ 
denotes the syzygy, or Heller, operator (see for example \cite[\S 20]{alp}).

\begin{thm}
\label{thm:supermain}
Let $V$ be a finitely generated indecomposable $kS_4$-module with stable endomorphism ring $k$, 
and let $\mathfrak{C}$ be the component of the stable Auslander-Reiten quiver of $kS_4$ 
containing $V$. 
\begin{enumerate}
\item[i.] Suppose $\mathfrak{C}$ or $\Omega(\mathfrak{C})$ contains $T_0$. 
Then $\mathfrak{C}$ is of type $\mathbb{Z}A_\infty^\infty$, and all modules in $\mathfrak{C}\cup
\Omega(\mathfrak{C})$ have stable endomorphism ring equal to $k$. In this case,
$R(S_4,V)$ is isomorphic to  $W [\mathbb{Z}/2]$.
\item[ii.] Suppose $\mathfrak{C}$ contains $T_1$. Then $\mathfrak{C}$ is of type 
$\mathbb{Z}A_\infty^\infty$, and $\mathfrak{C}=\Omega(\mathfrak{C})$. In this case,
$V$ is isomorphic to $\Omega^j(T_1)$ for some integer $j$, and 
$R(S_4,V)$ is isomorphic to $W[t]/(t^2,2t)$.
\item[iii.] Suppose $\mathfrak{C}$ contains a uniserial module $Y$ of length $3$ with non-isomorphic
top and socle. Then $\mathfrak{C}$ is a $3$-tube with $Y$ belonging to its boundary, and 
$\mathfrak{C}=\Omega(\mathfrak{C})$. In this case,
$V$ is isomorphic to $Y$, $\Omega^2(Y)$ or $\Omega^4(Y)$, and $R(S_4,V)$ is isomorphic to $k$.
\end{enumerate}
The only components of the stable Auslander-Reiten quiver of $kS_4$ containing modules with 
stable endomorphism ring $k$ are the ones in $(i) - (iii)$.
\end{thm}

Note that in parts (i) and (iii), $R(S_4,V)$ is a complete intersection, whereas in part (ii) this is 
not the case. In all three parts, 
$R(S_4,V)$ is isomorphic to a subquotient ring of $WD_8$ when $D_8$ is a dihedral group of order 
$8$.
In particular, this gives a positive answer in case of the unique $2$-modular block of $S_4$ to a 
question raised by the first author and Chinburg in \cite[Question 1.1]{bc} whether the universal 
deformation ring of a representation of a finite group with stable endomorphism ring $k$ is
always isomorphic to a subquotient ring of the group ring over $W$ of a defect group of the 
modular block  associated to the representation.

The paper is organized as follows. In \S \ref{s:prelim}, we give some background on universal deformation rings. In \S \ref{s:ks4}, we state the properties of $kS_4$ we need to prove
Theorem \ref{thm:supermain}. In particular, we provide the necessary results concerning the
indecomposable modules and the stable Auslander-Reiten quiver of $kS_4$.
 In \S \ref{s:c0}--\S\ref{s:stableend},
 we prove Theorem
 \ref{thm:supermain}.
 
 Part of this paper constitutes the Ph.D. thesis of the second author under the supervision
 of the first author \cite{llosent}.


\section{Preliminaries}
\label{s:prelim}
\setcounter{equation}{0}
\setcounter{figure}{0}

Let $k$ be an algebraically closed field of characteristic $p>0$, let $W$ be the ring of infinite Witt 
vectors over $k$ and let $F$ be the fraction field of $W$. Let ${\mathcal{C}}$ be the category of 
all complete local commutative Noetherian rings with residue field $k$. The morphisms in 
${\mathcal{C}}$ are continuous $W$-algebra homomorphisms which induce the identity map on $k$. 

Suppose $G$ is a finite group and $V$ is a finitely generated $kG$-module. 
A lift of $V$ over an object $R$ in ${\mathcal{C}}$ is a finitely generated $RG$-module $M$ which
is free over $R$ together with a $kG$-module isomorphism 
$\phi:k\otimes_R M\to V$. Two lifts $(M,\phi)$ and $(M',\phi')$
of $V$ over $R$ are isomorphic if there is an $RG$-module isomorphism 
$\alpha:M\to M'$ such that $\phi'\circ (1\otimes\alpha) = \phi$. 
The isomorphism class of a lift of $V$ 
over $R$ is called a deformation of $V$ over $R$, and the set of such deformations is denoted by 
$\mathrm{Def}_G(V,R)$. The deformation functor ${F}_V:{\mathcal{C}} \to \mathrm{Sets}$
is defined to be the covariant functor which sends an object $R$ in ${\mathcal{C}}$ to 
$\mathrm{Def}_G(V,R)$.

Suppose there exists an object $R(G,V)$ in ${\mathcal{C}}$ and a lift $(U(G,V),\phi_U)$ 
of $V$ over $R(G,V)$ satisfying the following.
For each $R$ in ${\mathcal{C}}$ and for each lift $(M,\phi)$ of $V$ over $R$ there is a unique 
morphism $\alpha:R(G,V)\to R$ in ${\mathcal{C}}$ such that the lift $(M,\phi)$ is isomorphic to 
$(R\otimes_{R(G,V),\alpha}U(G,V),\phi_U)$ where we identify 
$k\otimes_R(R\otimes_{R(G,V),\alpha}U(G,V))$ with $k\otimes_{R(G,V)}U(G,V)$.
Then $R(G,V)$ is called the universal deformation ring of $V$ and the isomorphism class
of the lift $(U(G,V),\phi_U)$ is called the universal deformation of $V$. In other words, $R(G,V)$ represents
the functor ${F}_V$ in the sense that ${F}_V$ is naturally isomorphic to 
$\mathrm{Hom}_{{\mathcal{C}}}(R(G,V),-)$. For more information on deformation rings see 
\cite{lendesmit} and \cite{maz1}.

The following four results were proved in \cite{bc} and in \cite{3sim}, respectively. As before, $\Omega$ 
denotes the syzygy, or Heller, operator for $kG$ (see for example \cite[\S 20]{alp}).

\begin{prop}
\label{prop:stablend}
{\rm (\cite[Prop. 2.1]{bc}).}
Suppose $V$ is a finitely generated $kG$-module with stable endomorphism ring 
$\underline{\mathrm{End}}_{kG}(V)=k$.  Then $V$ has  a universal deformation ring $R(G,V)$.
\end{prop}

\begin{lemma} 
\label{lem:defhelp}
{\rm (\cite[Cors. 2.5 and 2.8]{bc}).}
Let $V$ be a finitely generated $kG$-module with stable endomorphism ring 
$\underline{\mathrm{End}}_{kG}(V)=k$.
\begin{enumerate}
\item[i.] Then $\underline{\mathrm{End}}_{kG}(\Omega(V))=k$, and $R(G,V)$ and $R(G,\Omega(V))$ 
are isomorphic.
\item[ii.] There is a non-projective indecomposable $kG$-module $V_0$ $($unique up to
isomorphism$)$ such that $\underline{\mathrm{End}}_{kG}(V_0)=k$, $V$ is isomorphic to 
$V_0\oplus P$ for some projective $kG$-module $P$, and $R(G,V)$ and $R(G,V_0)$ are 
isomorphic.
\end{enumerate}
\end{lemma}

\begin{prop}
\label{prop:induceddef}
{\rm (\cite[Prop. 2.1.3]{3sim}).}
Let $L$ be a subgroup of $G$ and let $U$ be a finitely generated indecomposable $kL$-module with 
$\underline{\mathrm{End}}_{kL}(U)= k$. Suppose there exists an indecomposable $kG$-module $V$ 
with $\underline{\mathrm{End}}_{kG}(V)= k$ and a projective $kG$-module $P$ such that 
$\mathrm{Ind}_L^G U = V\oplus P$.
Assume further that
$$\mathrm{dim}_k \mathrm{Ext}_{kL}^1(U,U) = \mathrm{dim}_k \mathrm{Ext}_{kG}^1(V,V).$$
Then $R(G,V)$ is isomorphic to $R(L,U)$.
\end{prop}

\begin{lemma}
\label{lem:Wlift}
{\rm (\cite[Lemma 2.3.2]{3sim}).}
Let $V$ be a finitely generated $kG$-module such that there is a non-split short exact sequence of 
$kG$-modules
$$0\to Y_2\to V\to Y_1\to 0$$
with $\mathrm{Ext}^1_{kG}(Y_1,Y_2)=k$.
Suppose that for $i\in\{1,2\}$, there exists a $WG$-module $X_i$ which defines a lift of $Y_i$ over $W$. 
Suppose further that 
$$\mathrm{dim}_F\;\mathrm{Hom}_{FG}(F\otimes_WX_1,F\otimes_WX_2) =
\mathrm{dim}_k\;\mathrm{Hom}_{kG}(Y_1,Y_2)-1.$$
Then there exists a $WG$-module $X$ which defines a lift of $V$ over $W$.
\end{lemma}


\section{The group ring $kS_4$}
\label{s:ks4}
\setcounter{equation}{0}
\setcounter{figure}{0}

Let $k$ be an algebraically closed field of characteristic $2$, let $W$ be the ring of infinite
Witt vectors over $k$ and let $F$ be the fraction field of $W$.
Then $kS_4$ is its own
principal block, and the defect groups are the Sylow $2$-subgroups of $S_4$ which
are dihedral groups of order $8$. By \cite[\S V.2.5.1]{erd}, $kS_4$ is Morita equivalent to the special 
biserial algebra $\Lambda=kQ/I$ where
\begin{equation}
\label{eq:qi}
\xymatrix @R=-.2pc {
&0&1\\
Q=\quad& \ar@(ul,dl)_{\alpha} \bullet \ar@<.8ex>[r]^{\beta} &\bullet\ar@<.9ex>[l]^{\gamma}
\ar@(ur,dr)^{\eta}&
\quad\mbox{ and }\quad  I=\langle \alpha^2,\eta\beta,\gamma\eta,\beta\gamma,
\gamma\beta\alpha-\alpha\gamma\beta,\eta^2-\beta\alpha\gamma\rangle.}
\end{equation}
We denote the irreducible $\Lambda$-modules by $S_0$ and $S_1$, or, using short-hand, by $0$
and $1$. Then $S_0$ corresponds to the trivial simple $kS_4$-module $T_0$, and
$S_1$ corresponds to the two-dimensional simple $kS_4$-module $T_1$.
The radical series of the projective indecomposable $\Lambda$-modules (and hence of the projective indecomposable $kS_4$-modules) can be described by the following pictures:
\begin{equation}
\label{eq:proj}
P_0=\begin{array}{c@{\hspace*{.5ex}}c@{\hspace*{.5ex}}c}&0&\\
0&&1\\1&&0\\&0&\end{array},\qquad 
P_1=\begin{array}{c@{\hspace*{.5ex}}c@{\hspace*{.5ex}}c}&1&\\ 1&&
\begin{array}{c}0\\0\end{array}\\&1&\end{array}.
\end{equation}
The field $F$ is a splitting field for $S_4$, and the decomposition matrix of $kS_4$ has the form
\begin{equation}
\label{eq:decoms4}
\begin{array}{cc}
&\begin{array}{c@{}c}\varphi_0\,&\,\varphi_1\end{array}\\[1ex]
\begin{array}{c}\chi_1\\ \chi_2\\ \chi_3 \\ \chi_4\\ \chi_5\end{array} &
\left[\begin{array}{cc}1&0\\1&0\\1&1\\1&1\\0&1\end{array}\right].
\end{array}
\end{equation}

Because $\Lambda$ is a special biserial algebra, all indecomposable non-projective 
$\Lambda$-modules
are either  string or band modules. Since we will use this description extensively,
we now give a brief introduction. For more background material, we refer to \cite{buri}.

For each arrow $\alpha,\beta,\gamma,\eta$ in $Q$, we define formal inverses
$\alpha^{-1},\beta^{-1},\gamma^{-1},\eta^{-1}$ with starting points $s(\alpha^{-1})=0=
s(\gamma^{-1})$ and $s(\beta^{-1})=1=s(\eta^{-1})$ and end points $e(\alpha^{-1})=0=
e(\beta^{-1})$ and $e(\gamma^{-1})=1=e(\eta^{-1})$.
A word $w$ is a sequence $w_1\cdots w_n$, where $w_i$ is either an arrow or a formal inverse such 
that $s(w_i)=e(w_{i+1})$ for $1\leq i \leq n-1$. Define $s(w)=s(w_n)$, $e(w)=e(w_1)$ and 
$w^{-1}=w_n^{-1}\cdots w_1^{-1}$. There are also empty words $1_0$ and $1_1$ of length $0$ with 
$e(1_0)=0=s(1_0)$, $e(1_1)=1=s(1_1)$ and $(1_0)^{-1}=1_0$, $(1_1)^{-1}=1_1$. 
Denote the set of all words by $\mathcal{W}$, and the set of all 
non-empty words $w$ with $e(w)=s(w)$ by ${\mathcal{W}}_r$. 
Let $J=\{\alpha^2,\eta\beta,\beta\gamma,\gamma\eta,\eta^2,
\gamma\beta\alpha,\alpha\gamma\beta,\beta\alpha\gamma\}$.

\begin{dfn}
\label{def:strings}
Let $\sim_s$ be the equivalence relation on $\mathcal{W}$ with $w\sim_s w'$ if and only if $w=w'$ or 
$w^{-1}=w'$. Then strings are representatives $w\in{\mathcal{W}}$ of the equivalence classes under 
$\sim_s$ with the following property: Either $w=1_u$ for $u\in\{0,1\}$, or $w=w_1\cdots w_n$ where 
$w_i \neq w_{i+1}^{-1}$ for $1\leq i\leq n-1$ and no subword of $w$ or its formal inverse belongs to 
$J$.

Let $C=w_1\cdots w_n$ be a string of length $n$. Then there exists an
indecomposable $\Lambda$-module $M(C)$,
called the string module $M(C)$ corresponding to the string $C$, which can be described as follows.
There is a $k$-basis $\{z_0,z_1,\ldots, z_n\}$ of $M(C)$ such that the action of $\Lambda$ on $M(C)$ 
is given by the following representation $\varphi_C:\Lambda\to\mathrm{Mat}(n+1,k)$. 
Let $v(i)=e(w_{i+1})$ for $0\leq i\leq n-1$ and $v(n)=s(w_n)$. Then for each vertex $u\in\{0,1\}$ and for 
each arrow $\xi\in\{\alpha,\beta,\gamma,\eta\}$ in $Q$
$$\varphi_C(u)(z_i) = \left\{ \begin{array}{c@{\quad,\quad}l}
z_i & \mbox{if $v(i)=u$}\\ 0 & \mbox{else} \end{array} \right\}\; \mbox{ and } \;
\varphi_C(\xi)(z_i) = \left\{ \begin{array}{c@{\quad,\quad}l}
z_{i-1} & \mbox{if $w_i=\xi$}\\ z_{i+1} & \mbox{if $w_{i+1}=\xi^{-1}$}\\
0 & \mbox{else}
\end{array} \right\} .$$
We will call $\varphi_C$ the canonical representation and $\{z_0,z_1,\ldots,z_n\}$ the canonical $k$-basis for $M(C)$ relative to the representative $C$. Note that $M(C)\cong M(C^{-1})$. 

The string modules for the empty words are isomorphic to the simple $\Lambda$-modules,
namely $M(1_0)\cong S_0$ and $M(1_1)\cong S_1$.
\end{dfn}

\begin{dfn}
\label{def:bands}
Let $w=w_1\cdots w_n\in {\mathcal{W}}_r$. Then, for $0\leq i\leq n-1$, the $i$-th rotation of $w$ is 
defined to be the word $\rho_i(w)=w_{i+1}\cdots w_n w_1 \cdots w_i$. Let $\sim_r$ be the 
equivalence relation on ${\mathcal{W}}_r$ such that
$w\sim_r w'$ if and only if $w=\rho_i(w')$ for some $i$ or $w^{-1}=\rho_j(w')$ for some $j$. 
Then bands are representatives $w\in {\mathcal{W}}_r$ of the equivalence classes under 
$\sim_r$ with the following property: 
$w=w_1\cdots w_n$, $n\ge 1$, with $w_i\neq w_{i+1}^{-1}$ and $w_n\neq w_1^{-1}$, such that 
$w$ is not a power of a smaller word, and, for all positive integers $m$, no subword of $w^m$ 
or its formal inverse belongs to $J$.

Let $B=w_1\cdots w_n$ be a band of length $n$. Then for each integer $m>0$ and each 
$\lambda\in k^*$ there exists an indecomposable $\Lambda$-module $M(B,\lambda,m)$ which is
called the band module corresponding to the band $B$, $\lambda$ and $m$.
Note that for all $i,j$
$$M(B,\lambda,m)\cong M(\rho_i(B),\lambda,m)\cong M(\rho_j(B)^{-1},\lambda,m).$$ 
\end{dfn}

Each component of the stable Auslander-Reiten quiver of $\Lambda$ consists either entirely of string 
modules or entirely of band modules. The band modules all lie in $1$-tubes. The components
consisting of string modules are one $1$-tube, one $3$-tube and infinitely many non-periodic 
components of type $\mathbb{Z}A_\infty^\infty$. 

The irreducible morphisms between string modules can be described using hooks and cohooks.  
For our algebra $\Lambda$, these are defined as follows. Let $\mathcal{M}$ be the set of 
maximal directed strings, i.e. $\mathcal{M}=\{\gamma\beta,\beta\alpha,\alpha\gamma,\eta\}$.

\begin{dfn}
\label{def:arcomps}
Let $S$ be a string.
We say that $S$ starts on a peak (resp. starts in a deep) if $S=S'C$ (resp. $S=S'C^{-1}$)
for some string $C$ in $\mathcal{M}$. Dually, we say that $S$ ends on a peak (resp. ends in a deep)
if $S=D^{-1}S''$ (resp. $S=DS''$) for some string $D$ in $\mathcal{M}$.

If $S$ does not start on a peak (resp. does not start in a deep), there is a unique arrow $\zeta$ and a 
unique $M\in\mathcal{M}$ such that $S_h=S\zeta M^{-1}$ (resp. $S_c=S\zeta^{-1}M$) is a string. 
We say $S_h$ (resp. $S_c$) is obtained from $S$ by adding a hook (resp. a cohook) on 
the right side.

Dually, if $S$ does not end on a peak (resp. does not end in a deep), there is a unique arrow $\xi$ 
and a unique $N\in\mathcal{M}$ such that ${}_hS=M\xi^{-1}S$ (resp. ${}_cS=N^{-1}\xi S$)
is a string.  We say ${}_hS$ (resp. ${}_cS$) is obtained from $S$ by adding a hook (resp. a cohook) 
on the left side.
\end{dfn}

All irreducible morphisms between string modules are either canonical injections
$M(S)\to M(S_h)$, $M(S)\to M({}_hS)$,
or canonical projections
$M(S_c)\to M(S)$, $M({}_cS)\to M(S)$.

In particular, since none of the projective $\Lambda$-modules is uniserial, we get the following result.
Suppose $S$ is a string of minimal positive length in a component of the stable Auslander-Reiten
quiver of $\Lambda$ of type $\mathbb{Z}A_\infty^\infty$. 
Then near $M(S)$ the stable Auslander-Reiten component looks as in Figure \ref{fig:arbcomp}.
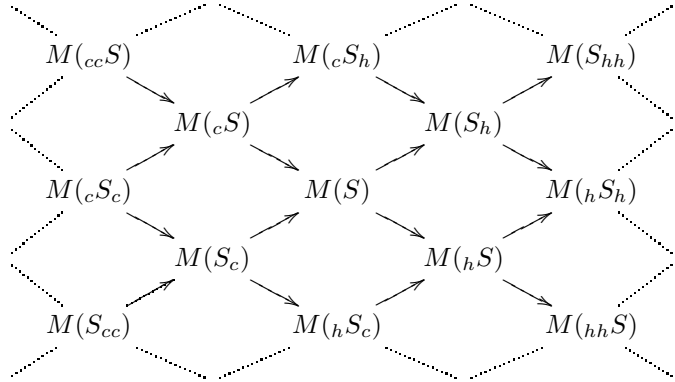
\begin{figure}[ht] \hrule \caption{\label{fig:arbcomp} The stable Auslander-Reiten component 
near $M(S)$.}
$$\xymatrix @-1.2pc{
&&&&&&\\
&M({}_{cc}S)\ar[rd]\ar@{.}[lu]\ar@{.}[ld]\ar@{.}[ru]&&
M({}_cS_h)\ar[rd]\ar@{.}[ru]\ar@{.}[lu]
&&M(S_{hh})\ar@{.}[ru]\ar@{.}[rd]\ar@{.}[lu]&\\
&&M({}_{c}S)\ar[rd]\ar[ru]&&M(S_{h})\ar[rd]\ar[ru]&&\\
&M({}_cS_c)\ar[ru]\ar[rd]\ar@{.}[lu]\ar@{.}[ld]&&M(S)\ar[ru]\ar[rd]&&M({}_hS_h)\ar@{.}[ru]\ar@{.}[rd]&\\
&&M(S_{c})\ar[ru]\ar[rd]\ar@{.}[ld]&&M({}_{h}S)\ar[ru]\ar[rd]&&\\
&M(S_{cc})\ar[ru]\ar@{.}[ld]\ar@{.}[lu]\ar@{.}[rd]&&
M({}_hS_c)\ar[ru]\ar@{.}[rd]\ar@{.}[ld]
&&M({}_{hh}S)\ar@{.}[rd]\ar@{.}[ru]\ar@{.}[ld]&\\
&&&&&&\\&&&&&&
}$$
\hrule
\end{figure}

In \cite{krau}, all homomorphisms between string and band modules have been determined.
The following remark describes the homomorphisms between string modules using the
canonical bases defined in Definition \ref{def:strings}.

\begin{rem}
\label{rem:stringhoms}
Let $M(S)$ (resp. $M(T)$) be a string module with canonical $k$-basis $\{x_u\}_{u=0}^m$ (resp. $\{y_v\}_{v=0}^n$) relative to the representative $S$ (resp. $T$).
Suppose $C$ is a string such that
\begin{enumerate}
\item[i.] $S\sim_s S'CS''$ with ($S'$ of length $0$ or $S'=\hat{S}'\zeta_1$) and ($S''$ of length $0$ or 
$S''=\zeta_2^{-1} \hat{S}''$), where $S',\hat{S}',S'',\hat{S}''$ are strings and $\zeta_1,\zeta_2$ are arrows in $Q$; 
and
\item[ii.] $T\sim_sT'CT''$ with ($T'$ of length $0$ or $T'=\hat{T}'\xi_1^{-1} $) and ($T''$ of length $0$ or 
$T''=\xi_2 \hat{T}''$), where $T',\hat{T}',T'',\hat{T}''$ are strings and $\xi_1,\xi_2$ are arrows in $Q$.
\end{enumerate} 
Then there exists a non-zero $\Lambda$-module homomorphism $\sigma_C:M(S)\to M(T)$ which factors through $M(C)$ and which sends 
each element of $\{x_u\}_{u=0}^m$ either to zero or to an element of $\{y_v\}_{v=0}^n$,
according to the relative position of $C$ in $S$ and $T$, respectively.
If e.g. $S=s_1s_2\cdots s_m$, $T=t_1t_2\cdots t_n$, and $C=s_{i+1}s_{i+2}\cdots s_{i+\ell} = t_{j+\ell}^{-1}t_{j+\ell-1}^{-1}\cdots t_{j+1}^{-1}$,
then 
$$\sigma_C(x_{i+t})=y_{j+\ell-t} \mbox{ for } 0\le t\le \ell, \mbox{ and } \sigma_C(x_u)=0\mbox{ for all other $u$.}$$
Note that there may be several choices of $S',S''$ (resp. $T',T''$) in (i) (resp. (ii)). In other words, there 
may be several $k$-linearly independent homomorphisms factoring through $M(C)$. By \cite{krau}, every 
$\Lambda$-module homomorphism $\sigma:M(S)\to M(T)$ is a $k$-linear combination of 
homomorphisms which factor through string modules corresponding to strings $C$ satisfying 
(i) and (ii). 
\end{rem}

The next result is an easy consequence of 
\cite{krau}
and the symmetric shapes of the radical series of $P_0$ and $P_1$ as given in 
$(\ref{eq:proj})$.

\begin{lemma}
\label{lem:upsidedown}
Let $\nu$ be the bijection on the arrows in $Q$ which fixes $\alpha$ and 
$\eta$ and interchanges $\beta$ and $\gamma$.
Let $S=w_1w_2\cdots w_n$ be a string of length $n\ge 1$, and define 
$\nu(S)=\nu(w_1)^{-1}\nu(w_2)^{-1}
\cdots \nu(w_n)^{-1}$. Then $\nu(S)$ is a string,
$\mathrm{dim}_k\,\underline{\mathrm{End}}_\Lambda(M(S))=
\mathrm{dim}_k\,\underline{\mathrm{End}}_\Lambda(M(\nu(S)))$,
and $\mathrm{dim}_k\,\underline{\mathrm{Ext}}^1_\Lambda(M(S),M(S))=
\mathrm{dim}_k\,\underline{\mathrm{Ext}}^1_\Lambda(M(\nu(S)),M(\nu(S)))$.
\end{lemma}


\section{The stable Auslander-Reiten components of $kS_4$ containing $T_0$  and $\Omega(T_0)$}
\label{s:c0}
\setcounter{equation}{0}
\setcounter{figure}{0}

In this section, we prove part (i) of Theorem \ref{thm:supermain}. 
Let $\Lambda=kQ/I$ where $k$, $Q$ and $I$ are as in $(\ref{eq:qi})$, i.e. $\Lambda$ is Morita
equivalent to $kS_4$.

Let $\mathfrak{C}_0$ be the component of the stable Auslander-Reiten quiver of $\Lambda$
containing $S_0$, and let $U$ be a $\Lambda$-module belonging to $\mathfrak{C}_0\cup\Omega(\mathfrak{C}_0)$. 
Using the description of the components of the stable Auslander-Reiten quiver of $\Lambda$ as in 
\S\ref{s:ks4}, we see that $\mathfrak{C}_0$ is of type $\mathbb{Z}A_\infty^\infty$. Moreover, 
using hooks and cohooks (see Definition \ref{def:arcomps}) we obtain the following.
There is an $i\in\mathbb{Z}$ such that $\Omega^i(U)$ is isomorphic to one of the following string modules: 
\begin{equation}
\label{eq:c0mods}
S_0,\quad A_n = M((\alpha\beta^{-1}\gamma^{-1})^n),\quad B_n=M((\alpha^{-1}\gamma\beta)^n),
\quad n\ge 1.
\end{equation}

\begin{lemma}
\label{lem:c0endos}
For $n\ge 1$, let $V_n$ be an indecomposable $kS_4$-module which, under the Morita equivalence,
corresponds to the $\Lambda$-module $A_n$ or $B_n$ from $(\ref{eq:c0mods})$. 
Then $\underline{\mathrm{End}}_{kS_4}(V_n)\cong k\cong \mathrm{Ext}^1_{kS_4}(V_n,V_n)$.
\end{lemma}

\begin{proof}
We prove this for $V_n$ corresponding to the $\Lambda$-module $B_n$.
The case when $V_n$ corresponds to the $\Lambda$-module $A_n$ follows then from Lemma 
\ref{lem:upsidedown}.

We first show how the results in \cite{AC} imply
that $V_n$ is endo-trivial, in the sense that the $kS_4$-module 
$\mathrm{End}_k(V_n)$ is isomorphic to a direct sum of the trivial simple $kS_4$-module
$T_0=k$ and a projective $kS_4$-module. Note that if $M$ and $N$ are endo-trivial $kS_4$-modules,
then so are $M\otimes_k N$ and $\Omega^i(M)$ for all $i\in\mathbb{Z}$. 
Consider the almost split sequence ending in $T_0$
\begin{equation}
\label{eq:oh1}
\xymatrix{0\ar[r]& \Omega^2(T_0) \ar[r]& V_1\oplus W_1 \ar[r]& T_0\ar[r]& 0}
\end{equation}
where, by our assumptions, $V_1$ corresponds to the $\Lambda$-module $B_1$, and $W_1$ 
corresponds to the $\Lambda$-module $\Omega^2(A_1)$. 
Since $\Omega^{-1}(V_1)\oplus \Omega^{-1}(W_1)$ is isomorphic to the heart
$\mathrm{rad}(P_{T_0})/\mathrm{soc}(P_{T_0})$
of the projective indecomposable $kS_4$-module $P_{T_0}$ with top $T_0$, it follows from \cite[Lemma 5.4]{AC} that
$V_1$ and $W_1$ are endo-trivial $kS_4$-modules. 
Given an indecomposable $kS_4$-module $M$ of odd $k$-dimension, it follows from
\cite[Thm. 3.6 and Cor. 4.7]{AC} that 
\begin{equation}
\label{eq:oh2}
\xymatrix{0\ar[r]& \Omega^2(T_0)\otimes_k M\ar[r]&  (V_1\otimes_k M)\oplus (W_1\otimes_k M) \ar[r] 
&M\ar[r]& 0}
\end{equation}
is the almost split sequence ending in $M$ modulo projective direct summands.
Applying this inductively to $M=V_n$ and using that $V_{n+1}$ is a direct
summand of the middle term of the almost split sequence ending in $V_n$, we obtain that 
$V_n$ is endo-trivial for all $n\ge 1$. 

Since $V_n$ is endo-trivial, it follows that $\underline{\mathrm{End}}_{kS_4}(V_n)\cong k$ and that
$$\mathrm{Ext}_{kS_4}^1(V_n,V_n)\cong 
\HH^1(S_4,\mathrm{End}_k(V_n))
\cong \HH^1(S_4,T_0)\cong \mathrm{Ext}^1_{kS_4}(T_0,T_0)\cong k.$$
\end{proof}

\begin{lemma}
\label{lem:c0udr}
For $n\ge 1$, let $V_n$ be an indecomposable $kS_4$-module which, under the Morita equivalence,
corresponds to the $\Lambda$-module $A_n$ or $B_n$ from $(\ref{eq:c0mods})$. 
Then $R(S_4,V_n)$ is isomorphic to the group ring $W[\mathbb{Z}/2]$.
\end{lemma}

\begin{proof}
We prove this for $V_n$ corresponding to the $\Lambda$-module $A_n$, the case
of $V_n$ corresponding to  $B_n$ being similar.

We use the description in \cite[Lemma V.2.5]{erd} of the Morita equivalence between $kS_4$ and 
$\Lambda=kQ/I$ with $Q$ and $I$ as in $(\ref{eq:qi})$. From this it follows that we can take
the arrow $\alpha$ to be $(1+h)e_0$ where $h\in S_4$ is a transposition and
$e_0$ is the idempotent corresponding to the
projective indecomposable $kS_4$-module $P_{T_0}$ with top $T_0$.
Let $C$ be the cyclic
subgroup of $S_4$ of order two generated by $h$. Let $T_{00}$ be the uniserial 
$kS_4$-module with radical series $T_0,T_0$. Because of our choice of $h$, it follows that
restricting the action of $S_4$ on $T_{00}$ to $C$ results in a $2$-dimensional $kC$-module with
non-trivial action. Considering a representation of the simple $kS_4$-module $T_1$,
it follows that  restricting the action of $S_4$ on $T_1$ to $C$ also results in
a $2$-dimensional $kC$-module with non-trivial action. Hence
\begin{equation}
\label{eq:restrict}
\mathrm{Res}^{S_4}_C\,T_{00}\cong kC\cong \mathrm{Res}^{S_4}_C\,T_1.
\end{equation}

Let $A_0=S_0$. For $n\ge 1$, we have a short exact sequence of $\Lambda$-modules
$$\xymatrix{
0\ar[r]& A_{n-1} \ar[r] & A_n \ar[r] & Y \ar[r]& 0}$$
where $Y=M(\alpha\beta^{-1})$. If $Z$ is the $kS_4$-module
corresponding to $Y$, then $(\ref{eq:restrict})$ implies that 
$$\mathrm{Res}^{S_4}_C\,Z\cong kC\oplus kC.$$
Hence it follows by induction that 
$$\mathrm{Res}^{S_4}_C\,V_n\cong k\oplus (kC)^{2n}$$
where  $k$ stands for the trivial simple $kC$-module. 
Therefore, $\mathrm{Res}^{S_4}_C\,V_n$ is a $kC$-module with stable endomorphism ring $k$
which has  universal deformation ring $R(C,\mathrm{Res}^{S_4}_{C}V_n)\cong W[\mathbb{Z}/2]$.
Let $(U_{V_n,C},\phi_{n,C})$  be a universal lift of $\mathrm{Res}^{S_4}_{C}V_n$ over 
$W[\mathbb{Z}/2]$,
and let $(U_{V_n},\phi_n)$ be a universal lift of $V_n$ over $R(S_4,V_n)$. 
Then there exists a unique $W$-algebra homomorphism $f_n:W[\mathbb{Z}/2]\to R(S_4,V_n)$ in 
$\mathcal{C}$  such that the lift $(\mathrm{Res}^{S_4}_{C}\,U_{V_n},\mathrm{Res}^{S_4}_{C}\,\phi_n)$ 
is isomorphic to the lift $(R(S_4,V_n)\otimes_{W[\mathbb{Z}/2],f_n} U_{V_n,C},\,\phi_{n,C})$. 

We next show that $f_n$ is surjective for all $n\ge 1$. For this,
we consider all morphisms $\rho:R(S_4,V_n)\to k[\epsilon]/(\epsilon^2)$.
There is a non-split short exact sequence of $kS_4$-modules
\begin{equation}
\label{eq:ses}
\xymatrix{
0\ar[r] & V_n \ar[r] & (P_{T_0})^n \oplus T_{00}\ar[r] & V_n\ar[r] &0}.
\end{equation}
Hence $E_n=(P_{T_0})^n \oplus T_{00}$ defines a non-trivial lift of 
$V_n$ over $k[\epsilon]/(\epsilon^2)$. Since
$$\mathrm{Res}^{S_4}_C\,E_n\cong (kC)^{4n}\oplus kC,$$
$\mathrm{Res}^{S_4}_C\,E_n$ defines a non-trivial lift of $\mathrm{Res}^{S_4}_C\,V_n$
over $k[\epsilon]/(\epsilon^2)$.
Because $\mathrm{Ext}^1_{kS_4}(V_n,V_n)\cong k$ by Lemma \ref{lem:c0endos}, this implies
that as $\rho$ runs through the morphisms 
$R(S_4,V_n)\to k[\epsilon]/(\epsilon^2)$, the composition $\rho\circ f_n$ runs through the morphisms
$W[\mathbb{Z}/2]\to k[\epsilon]/(\epsilon^2)$. Hence $f_n$ is surjective.

We now use Lemma \ref{lem:Wlift} to prove that $V_n$ has always two non-isomorphic lifts over $W$. 
We first show that the $kS_4$-module $M$ corresponding to the $\Lambda$-string module
$M(\gamma\beta)$ has two non-isomorphic lifts over $W$. For this we consider
the representation $\varphi:S_4\to \mathrm{Mat}_4(W)$ which is the natural permutation 
representation of $S_4$, i.e. $\varphi$ is given by
$$\varphi((1,2))=\left(\begin{array}{cccc}0&1&0&0\\1&0&0&0\\0&0&1&0\\0&0&0&1\end{array}\right),
\quad\varphi((1,2,3,4))=\left(\begin{array}{cccc}0&0&0&1\\1&0&0&0\\0&1&0&0\\0&0&1&0
\end{array}\right).$$
Reducing these matrices modulo $2$, results in a representation $\overline{\varphi}:S_4\to \mathrm{Mat}_4(k)$ which is
the representation of a $4$-dimensional indecomposable $kS_4$-module corresponding
to the $\Lambda$-string module $M(\gamma\beta)$. We can order the ordinary
irreducible characters of $S_4$  in $(\ref{eq:decoms4})$ such that $\chi_1$ is the trivial character,
$\chi_2$ is the sign character and the $F$-character of $\varphi$ is equal to $\chi_1+\chi_3$.
Twisting $\varphi$ with the sign character $\chi_2$, we obtain a new representation 
$\varphi':S_4\to \mathrm{Mat}_4(W)$
with $F$-character $\chi_2+\chi_4$ which modulo $2$ defines a representation 
of $S_4$ over $k$ that is also equivalent to
$\overline{\varphi}$. This means that we get two non-isomorphic lifts $(X,\phi_X)$ and $(X',\phi_{X'})$
of $M$ over $W$ such that
the $F$-character of $X$ is  $\chi_1+\chi_3$ and the $F$-character of $X'$ is
$\chi_2+\chi_4$.

Let $V_0=T_0$. Then it follows from the decomposition matrix 
(\ref{eq:decoms4}) that there are two non-isomorphic $WS_4$-modules $U_{0,1}^W$ and 
$U_{0,2}^W$ which define lifts of $V_0$ over $W$ and whose $F$-characters are given by $\chi_1$ 
and $\chi_2$, respectively. 

Let $n\ge 1$. Then
$\mathrm{Hom}_{kS_4}(M,V_{n-1})\cong k^{n}$,
$\mathrm{Ext}^1_{kS_4}(M,V_{n-1})\cong k$, and
we have a short exact sequence of $kS_4$-modules
\begin{equation}
\label{eq:anotherses1}
\xymatrix{
0\ar[r]& V_{n-1} \ar[r] & V_n \ar[r] & M\ar[r]& 0}.
\end{equation}
Assume by induction that there are two non-isomorphic $WS_4$-modules
$U_{n-1,1}^W,$ and $U_{n-1,2}^W$ which define lifts of $V_{n-1}$ over $W$ such that
the $F$-character of  $U_{n-1,1}^W$ is
\begin{equation}
\label{eq:char2}
\left\{\begin{array}{ll} 
\chi_1+s(\chi_1+\chi_3) + s(\chi_2+\chi_4) & \mbox{ if }n-1=2s\\
\chi_1+ s(\chi_1+\chi_3) + (s+1)(\chi_2+\chi_4)& \mbox{ if }n-1=2s+1
\end{array}\right\}
\end{equation}
and the $F$-character of $U_{n-1,2}^W$ is 
\begin{equation}
\label{eq:char1}
\left\{\begin{array}{ll} 
\chi_2+s(\chi_1+\chi_3) + s(\chi_2+\chi_4) & \mbox{ if }n-1=2s\\
\chi_2+ (s+1)(\chi_1+\chi_3) + s(\chi_2+\chi_4)& \mbox{ if }n-1=2s+1
\end{array}\right\}.
\end{equation}
Let $X^F=F\otimes_WX$, ${X'}^F=F\otimes_WX'$, $U_{n-1,1}^F=F\otimes_W U_{n-1,1}^W$,
and $U_{n-1,2}^F=F\otimes_W U_{n-1,2}^W$.
Then if $n-1=2s$, 
$$\mathrm{dim}_F\,\mathrm{Hom}_{FS_4}({X'}^F,U_{n-1,1}^F)=s+s=n-1,$$
and if $n-1=2s+1$,
$$\mathrm{dim}_F\,\mathrm{Hom}_{FS_4}(X^F,U_{n-1,1}^F)=1+s+s=n-1.$$
Similarly, if $n-1=2s$ then 
$\mathrm{dim}_F\,\mathrm{Hom}_{FS_4}(X^F,U_{n-1,2}^F)=n-1$, 
and if $n-1=2s+1$ then
$\mathrm{dim}_F\,\mathrm{Hom}_{FS_4}({X'}^F,U_{n-1,2}^F)=n-1$.
Hence by Lemma \ref{lem:Wlift}, there are two non-isomorphic $WS_4$-modules $U_{n,1}^W$ and 
$U_{n,2}^W$ which define lifts of $V_n$ over $W$ and
whose $F$-characters are as in $(\ref{eq:char1})$ and $(\ref{eq:char2})$ when $n-1$
is replaced by $n$. 

Summarizing, we see that for each $n\ge 1$, we have a surjective morphism
$f_n:W[\mathbb{Z}/2]\to R(S_4,V_n)$ in $\mathcal{C}$ and there are two distinct morphisms
$R(S_4,V_n)\to W$ in $\mathcal{C}$.
Hence $\mathrm{Spec}(R(S_4,V_n))$ contains both points of the generic 
fiber of $\mathrm{Spec}(W[\mathbb{Z}/2])$.
Since the Zariski closure of these points is all of $\mathrm{Spec}(W[\mathbb{Z}/2])$,
this implies that $R(S_4,V_n)$ must be isomorphic to $W[\mathbb{Z}/2]$.
This completes the proof of Lemma \ref{lem:c0udr}.
\end{proof}

Since by \cite[\S 1.4]{maz1}, $R(S_4,T_0)\cong W[\mathbb{Z}/2]$,
part (i) of Theorem \ref{thm:supermain}  now follows from Lemmas \ref{lem:c0endos}, \ref{lem:c0udr} and
\ref{lem:defhelp}.


\section{The stable Auslander-Reiten component of $kS_4$ containing $T_1$}
\label{s:c1}
\setcounter{equation}{0}
\setcounter{figure}{0}

In this section, we prove part (ii) of Theorem \ref{thm:supermain}. 
As before, let $\Lambda=kQ/I$ where $k$, $Q$ and $I$ are as in $(\ref{eq:qi})$.

Let $\mathfrak{C}_1$ be the component of the stable Auslander-Reiten quiver of $\Lambda$
containing $S_1$. Using the description of the components of the stable Auslander-Reiten quiver 
of $\Lambda$ as in  \S\ref{s:ks4}, we see that $\mathfrak{C}_1$ is of type 
$\mathbb{Z}A_\infty^\infty$. Moreover near $S_1$, $\mathfrak{C}_1$ looks as in Figure
\ref{fig:c1}.
\begin{figure}[ht] \hrule \caption{\label{fig:c1} The stable Auslander-Reiten component 
$\mathfrak{C}_1$ near $S_1$.}
$$\xymatrix @-1.2pc{
&&&&&&\\
&&&S_{00}\ar[rd]&&&\\
&&\Omega(S_1)\ar[rd]\ar[ru]\ar@{.}[lu]&&\Omega^{-1}(S_1)\ar[rd]\ar@{.}[ru]&&\\
&\Omega^2(S_1)\ar[ru]\ar[rd]&&S_1\ar[ru]\ar[rd]\ar@{.}[ld]&&\Omega^{-2}(S_1)&\\
&&\Omega(S_{00})\ar[ru]\ar[rd]\ar@{.}[ld]&&\Omega^{-1}(S_{00})\ar[ru]\ar@{.}[rd]&&\\
&&&\Omega(S_{0011})\ar[ru]&&&\\&&&&&&
}$$
\hrule
\end{figure}
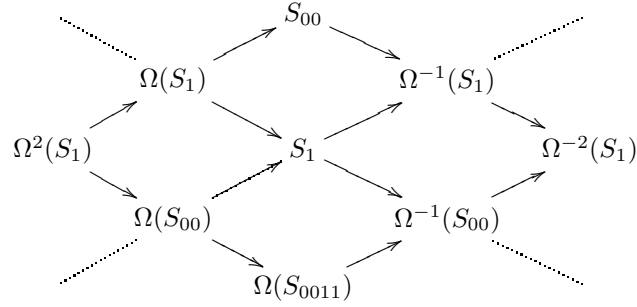
where $S_{00}$ is the uniserial $\Lambda$-module with radical series $S_0,S_0$ and 
$S_{0011}$ is the string module $M(\alpha^{-1}\gamma\eta^{-1})$. 
Hence $\mathfrak{C}_1=\Omega(\mathfrak{C}_1)$.

\begin{lemma}
\label{lem:c1endos}
The only $\Lambda$-modules belonging to $\mathfrak{C}_1$ which have stable endomorphism ring 
$k$ are of the form $\Omega^i(S_1)$ for some $i\in\mathbb{Z}$.
\end{lemma}

\begin{proof}
Let $U$ be a $\Lambda$-module belonging to $\mathfrak{C}_1$ which is not in the
$\Omega$-orbit of $S_1$.
Using  hooks and cohooks (see Definition \ref{def:arcomps}) we obtain the following because of
Figure \ref{fig:c1}.
There exist integers $i$ and $n\ge 1$ such that $\Omega^i(U)$ is isomorphic to one of the
following string modules: 
$$\begin{array}{c@{\;=\;}l}
C_{1,n}&M\left(\alpha^{-1}\left(\gamma\eta^{-1}\beta\alpha^{-1}\beta^{-1}\eta\gamma^{-1}\alpha^{-1}
\right)^{n-1}\right),\\[2ex]
C_{2,n}&M\left(\alpha^{-1}\left(\gamma\eta^{-1}\beta\alpha^{-1}\beta^{-1}\eta\gamma^{-1}\alpha^{-1}
\right)^{n-1}\gamma\eta^{-1}\right),\\[2ex]
C_{3,n} & M\left(\alpha^{-1}\left(\gamma\eta^{-1}\beta\alpha^{-1}\beta^{-1}\eta\gamma^{-1}\alpha^{-1}
\right)^{n-1}\gamma\eta^{-1}\beta\alpha^{-1}\beta^{-1}\right).
\end{array}$$
Let $\{x_r^{\ell,n}\}_{r=0}^{n_\ell}$ be the canonical $k$-basis for $C_{\ell,n}$, $\ell\in\{1,2,3\}$.
Then the $\Lambda$-module endomorphism of $C_{\ell,n}$ which sends $x_0^{\ell,n}$ to $x_1^{\ell,n}$
and all other $x_r^{\ell,n}$ to zero does not factor through a projective $\Lambda$-module. 
\end{proof}

Since it was proved in \cite[Thm. 2.3]{bc5} that $R(S_4,T_1)\cong W[t]/(t^2,2t)$,
part (ii) of Theorem \ref{thm:supermain}  now follows from Lemmas \ref{lem:c1endos} and
\ref{lem:defhelp}.


\section{The stable Auslander-Reiten component of $kS_4$  which is a $3$-tube}
\label{s:3tube}
\setcounter{equation}{0}
\setcounter{figure}{0}

In this section, we prove part (iii) of Theorem \ref{thm:supermain}. 
As before, let $\Lambda=kQ/I$ where $k$, $Q$ and $I$ are as in $(\ref{eq:qi})$.

Let $\mathfrak{C}$ be a component of the stable Auslander-Reiten quiver of $\Lambda$
containing a uniserial module $X$ of length $3$ with non-isomorphic top and socle.
Then $X$ is either $X_1$ with descending radical series $S_0,S_0,S_1$ or
$X_2$ with descending radical series $S_1,S_0,S_0$. In particular, $X_2=\Omega(X_1)=
\Omega^{-2}(X_1)$, which implies that $\mathfrak{C}$ is the unique $3$-tube of the
stable Auslander-Reiten quiver of $\Lambda$ with $X_1$ and $X_2$
belonging to its boundary, and $\mathfrak{C}=\Omega(\mathfrak{C})$. The third $\Lambda$-module
at the boundary of $\mathfrak{C}$ is $\Omega^2(X_1)$ which is the uniserial $\Lambda$-module 
with radical series $S_1,S_1$.

\begin{lemma}
\label{lem:3tubeudr}
The only $\Lambda$-modules belonging to $\mathfrak{C}$ which have stable endomorphism ring $k$
are $X_1$, $X_2=\Omega(X_1)$ and $\Omega^2(X_1)$. Moreover, if  $Y$ is a uniserial
$kS_4$-module of length $3$ with non-isomorphic top and socle and
$V\in\{Y,\Omega^2(Y),\Omega^4(Y)\}$, then $R(S_4,V)\cong k$.
\end{lemma}

\begin{proof}
Let $U$ be a $\Lambda$-module belonging to $\mathfrak{C}$ but not to its boundary. 
Using  hooks and cohooks (see Definition \ref{def:arcomps}) it follows that there exist integers 
$n\ge 1$ and
$i\in\{0,1,2\}$ such that $\Omega^i(U)$ is isomorphic to one of the following string modules: 
$$\begin{array}{c@{\;=\;}l}
X_{-1,n} & M\left( \left(\gamma^{-1}\alpha^{-1}\gamma\eta^{-1}\beta\alpha^{-1}\beta^{-1}\eta\right)^{n-1}
\gamma^{-1}\alpha^{-1}\gamma\eta^{-1}\right),\\[2ex]
X_{0,n} & M\left( \left(\gamma^{-1}\alpha^{-1}\gamma\eta^{-1}\beta\alpha^{-1}\beta^{-1}\eta\right)^{n-1}
\gamma^{-1}\alpha^{-1}\gamma\eta^{-1}\beta\alpha^{-1}\beta^{-1} \right),
\\[2ex]
X_{1,n} & M\left( \left(\gamma^{-1}\alpha^{-1}\gamma\eta^{-1}\beta\alpha^{-1}\beta^{-1}\eta\right)^n
\gamma^{-1}\alpha^{-1}\right).
\end{array}$$
Let $\{x_r^{\ell,n}\}_{r=0}^{8n-1+3\ell}$ be the canonical $k$-basis for $X_{\ell,n}$, $\ell\in\{-1,0,1\}$.
Then the $\Lambda$-module endomorphism of $X_{\ell,n}$ which sends $x_0^{\ell,n}$ to $x_4^{\ell,n}$,
and all other $x_r^{\ell,n}$ to zero does not factor through
a projective $\Lambda$-module. Hence the stable endomorphism ring of $X_{\ell,n}$ has
$k$-dimension at least $2$.

By Lemma \ref{lem:defhelp}, to prove $R(S_4,V)\cong k$, for $V$ as in the statement of the lemma, it is
sufficient to consider the case when $V=T_{11}$ is the uniserial $kS_4$-module with
radical series $T_1,T_1$. Let $A_4$ be the alternating group of order $12$ inside $S_4$.
There are three one-dimensional simple $kA_4$-modules $E_0$, $E_1$ and $E_2$ where
$E_0$ is the trivial simple $kA_4$-module. Consider the uniserial $kA_4$-module $E_{12}$ with
descending radical series $E_1,E_2$. Then $E_{12}$ lies at the boundary of a $3$-tube of the stable
Auslander-Reiten quiver of $kA_4$, and there is a non-split short exact sequence of $kA_4$-modules
$$\xymatrix{
0\ar[r] & E_2\ar[r] & E_{12} \ar[r] & E_1\ar[r] & 0.}$$
By \cite[Lemma 8.6(5)]{alp},  since $\mathrm{Ind}_{A_4}^{S_4}E_1\cong T_1\cong
\mathrm{Ind}_{A_4}^{S_4}E_2$, there is a non-split short exact 
sequence of $kS_4$-modules
$$\xymatrix{
0\ar[r] & T_1\ar[r] & \mathrm{Ind}_{A_4}^{S_4}E_{12} \ar[r] & T_1\ar[r] & 0}$$
which implies that $T_{11}\cong \mathrm{Ind}_{A_4}^{S_4}E_{12}$. By \cite[Prop. 3.4]{bl},
$\underline{\mathrm{End}}_{kA_4}(E_{12})=k$ and $R(A_4,E_{12})\cong k$. Since $\underline{
\mathrm{End}}_{kS_4}(T_{11})=k$ and
$$\mathrm{Ext}^1_{kS_4}(T_{11},T_{11})=0=\mathrm{Ext}^1_{kA_4}
(E_{12},E_{12}),$$
Proposition \ref{prop:induceddef}  implies that $R(S_4,T_{11})\cong R(A_4,E_{12})\cong k$.
\end{proof}


\section{Stable endomorphism rings}
\label{s:stableend}
\setcounter{equation}{0}
\setcounter{figure}{0}

In this section, we complete the proof of Theorem \ref{thm:supermain}, by determining which 
components of the stable Auslander-Reiten quiver of $kS_4$ contain modules with stable 
endomorphism ring $k$.
Recall that $kS_4$ is Morita equivalent to the special biserial algebra $\Lambda=kQ/I$ where 
$k$, $Q$ and $I$ are as in $(\ref{eq:qi})$.


We first consider all components of the stable Auslander-Reiten quiver of $\Lambda$ of type 
$\mathbb{Z}A_\infty^\infty$
and determine which of these components contain a module with stable endomorphism ring $k$.

\begin{prop}
\label{prop:endos}
Let $\Lambda=kQ/I$ where $Q$ and $I$ are as in $\S\ref{s:ks4}$. Then the components of the stable 
Auslander-Reiten quiver of $\Lambda$ of type $\mathbb{Z}A_\infty^\infty$ containing a module with 
stable endomorphism ring $k$ are precisely the components containing $S_0$, $\Omega(S_0)$ or 
$S_1$.
\end{prop}

\begin{proof}
Let $\mathfrak{C}$ be a component of type $\mathbb{Z}A_\infty^\infty$ of the stable Auslander-Reiten quiver of 
$\Lambda$ such that $\mathfrak{C}\cup\Omega(\mathfrak{C})$ contains no simple $\Lambda$-module.

Let $M(S)$ be a  string module of minimal length in $\mathfrak{C}$. In particular, $S$ 
cannot start or end in a peak (resp. in a  deep), and near $M(S)$ the component
$\mathfrak{C}$ looks as in Figure \ref{fig:arbcomp}. Hence it suffices to show that each of the 
$\Lambda$-modules
$M(S)$, $M(S_{h\cdots h})$ and $M(S_{c\cdots c})$ has an endomorphism which does not factor through a projective $\Lambda$-module
and which is not a scalar multiple of the identity homomorphism.

Using Lemma \ref{lem:upsidedown} and that $S$ cannot start or end in a peak (resp. in a  deep), 
it follows  that one only needs to consider $M(S)$ where 
\begin{equation}
\label{eq:oioioi}
S=\alpha^{-1}\gamma\cdots,\quad S=\gamma^{-1}\alpha\cdots,\quad \mbox{ or }\quad
S=\beta^{-1}\eta\gamma^{-1}\cdots.
\end{equation}
Note that $M(\alpha^{-1})$ lies in the same stable Auslander-Reiten component as $S_1$,
and hence has already been considered in Lemma \ref{lem:c1endos}.

One now goes through all the cases for $S$ in $(\ref{eq:oioioi})$ and determines suitable 
endomorphisms that do not factor through a projective $\Lambda$-module.
We provide a few examples.
Let $\{z_r\}_{r=0}^{\ell(S)}$ be the canonical $k$-basis of the string module $M(S)$ relative to the 
chosen representative $S$.  

If $S=\alpha^{-1}\gamma\cdots$ then the endomorphism sending $z_0$ to $z_1$ and all
other $z_r$ to zero does not factor through a projective $\Lambda$-module, and the same is true for
$M(S_{h\cdots h})$ and $M(S_{c\cdots c})$. 

If $S=\gamma^{-1}\alpha\cdots$, one needs to consider two possibilities, namely 
$S=\gamma^{-1}\alpha\beta^{-1}\cdots$ and $S=\gamma^{-1}\alpha\gamma\cdots$.
If $S=\beta^{-1}\eta\gamma^{-1}\cdots$, one needs to consider three possibilities, namely
$S=\beta^{-1}\eta\gamma^{-1}$, $S=\beta^{-1}\eta\gamma^{-1}\alpha\cdots$ and
$S=\beta^{-1}\eta\gamma^{-1}\alpha^{-1}\cdots$. For example, if 
$S=\beta^{-1}\eta\gamma^{-1}\alpha\cdots$ then 
the endomorphism sending $z_0$ to $z_3$ and all other $z_r$ to zero does not factor 
through a projective $\Lambda$-module, and the same is true for $M(S_{h\cdots h})$ and 
$M(S_{c\cdots c})$.
\end{proof}

Next, we consider all components of the stable Auslander-Reiten quiver which are one-tubes. 
Except for one $1$-tube consisting of string modules, all the other $1$-tubes consist of band modules.
It follows from \cite{krau} that if $B$ is a band, $\lambda\in k-\{0\}$ and $n\ge 2$ is an  
integer, then $\underline{\mathrm{End}}_\Lambda(M(B,\lambda,n))$ has $k$-dimension
at least $2$. Hence we can concentrate on the band modules $M(B,\lambda,1)$.
We need the following two definitions. 

\begin{dfn}
\label{def:bandendohelp}
Let  $B$ be a band for $\Lambda$, 
$\lambda\in k-\{0\}$ and let $M_{B,\lambda}=M(B,\lambda,1)$. 
Suppose $S$ is a string such that 
\begin{enumerate}
\item[i.] $B\sim_rST_1$ with $T_1=\xi_1^{-1} T'_1\xi_2$, where $T_1,T'_1$ are strings 
and $\xi_1,\xi_2$ are arrows in $Q$; and 
\item[ii.] $B\sim_r ST_2$ with $T_2=\zeta_1 T'_2\zeta_2^{-1}$, where $T_2,T'_2$ are 
strings and $\zeta_1,\zeta_2$ are arrows in $Q$.
\end{enumerate} 
Then by \cite{krau} there exists a non-zero endomorphism of $M_{B,\lambda}$ which factors 
through $M(S)$. We will call such an endomorphism to be of string type $S$. Note 
that there may be several choices of $T_1$ (resp. $T_2$) in (i) (resp. (ii)). In 
other words, there may be several $k$-linearly independent endomorphisms of string type $S$. By 
\cite{krau}, every endomorphism of $M_{B,\lambda}$ is a $k$-linear combination 
of the identity homomorphism and of endomorphisms of string type $S$ for suitable 
choices of strings $S$ satisfying (i) and (ii). 
\end{dfn}

\begin{dfn}
\label{def:piece}
Let  $B$ be a band for $\Lambda$. 
A word $S$ is called a  top-socle piece of $B$ if 
\begin{enumerate}
\item[a.] $S\in\{\alpha^{-1},\beta^{-1},\gamma^{-1},\eta^{-1},\alpha^{-1}\beta^{-1},
\beta^{-1}\gamma^{-1},\gamma^{-1}\alpha^{-1}\}$, and 
\item[b.] $B\sim_r ST$ for some string $T$ where $T=\xi T' \zeta$ and $\xi,\zeta$ are 
arrows in $Q$. 
\end{enumerate}
Note that $B\sim_r C_0C_1^{-1}C_2C_3^{-1}\cdots C_s^{-1}$ where $s\ge 1$ is odd and 
$C_0,C_1,\ldots,C_s$ are top-socle pieces of $B$. 
\end{dfn}

In the proof of the following proposition, we will often use the equality sign instead of the 
more precise $\sim_r$.

\begin{prop}
\label{prop:bands}
Let $\Lambda=kQ/I$ where $Q$ and $I$ are as in $\S\ref{s:ks4}$. Then each module $M$ which lies 
in a component of the stable Auslander-Reiten quiver of $\Lambda$ which is a $1$-tube satisfies 
$\mathrm{dim}_k\,\underline{\mathrm{End}}_\Lambda(M)\ge 2$.
\end{prop}

\begin{proof}
Consider first the $1$-tube consisting of string modules. If $X$ is a string module in
this $1$-tube, then $X=M(\beta^{-1}\gamma^{-1}(\alpha\beta^{-1}\gamma^{-1})^n)$ for some $n\ge 0$.
Let $\{z_r\}_{r=0}^{3n+2}$ be the canonical $k$-basis of the string module $X$.
Then  the endomorphism sending $z_0$ to $z_{3n+2}$ and all other
$z_r$ to zero does not factor through a projective $\Lambda$-module.

All other $1$-tubes consist of band modules, and it follows from
\cite{krau} that the only possible band modules with stable endomorphism ring $k$ are those
lying at the boundary. Hence one only needs to consider the band modules
$M_{B,\lambda}=M(B,\lambda,1)$ where $B$ is a band for $\Lambda$ and $\lambda\in k-\{0\}$.
Suppose there is a band $B$  such that $M_{B,\lambda}$ has stable endomorphism ring
equal to $k$.

Since every band contains either $\beta^{-1}\gamma^{-1}$ or $\eta^{-1}$ as top-socle piece,
one successively eliminates these two cases by determining endomorphisms of suitable string types
which do not factor through a projective $\Lambda$-module. We provide a few examples in each case.

Suppose first that $B$ contains $\beta^{-1}\gamma^{-1}$ as top-socle piece, i.e. 
$B=\beta^{-1}\gamma^{-1}C$ for some string $C=C_1^{-1}C_2C_3^{-1}\cdots C_s^{-1}$ where 
$s\ge 1$ is odd and $C_1,C_2,\ldots,C_s$ are top-socle pieces of $B$.
One needs to distinguish between $s=1$ and $s\ge 3$. 
If  $s\ge 3$ then, without loss of generality, $C_{1}^{-1}C_2\neq \alpha\beta^{-1}
\gamma^{-1}$. This leads to two cases, namely 
either $C_{1}^{-1}C_2 = \alpha\gamma\eta^{-1}$ or $C_{1}^{-1}C_2=
\alpha\beta^{-1}$. For example, if $C_{1}^{-1}C_2 = \alpha\gamma\eta^{-1}$ then $M_{B,\lambda}$
has an endomorphism of string type $\gamma$ or of string type $\gamma^{-1}\alpha$ which
does not factor through a projective $\Lambda$-module. 

Suppose now that $B$ contains $\eta^{-1}$ as top-socle piece but not $\beta^{-1}\gamma^{-1}$. 
One needs to distinguish between the case when $B=\eta^{-1}\beta\alpha^{-1}\gamma$ and 
the case when $B$ also contains $\alpha^{-1}\beta^{-1}$ as top-socle piece.  
If $B$ contains both $\eta^{-1}$ and $\alpha^{-1}\beta^{-1}$ as top-socle pieces then 
$B=\eta^{-1}\beta\alpha^{-1}\beta^{-1}\eta D$ for some string $D=D_1D_2^{-1}D_3\cdots D_u^{-1}$ 
where $u\ge 2$ is even and $D_1,D_2,\ldots,D_u$ are top-socle pieces of $B$. 
This leads to two subcases $D_1=\gamma^{-1}\alpha^{-1}$  or  $D_1=\gamma^{-1}$.
Each of these subcases leads to further subcases that have to be distinguished. For example,
if $D_1=\gamma^{-1}$, then either $D_2^{-1}=\alpha$ or $D_2^{-1}=\alpha\gamma$.
Concentrating on the case when $D_1D_2^{-1}=\gamma^{-1}\alpha$, one has $u\ge 4$ and, by our assumptions, $D_3D_4^{-1}=\beta^{-1}\eta$. Then $M_{B,\lambda}$
has an endomorphism of string type $\beta^{-1}\eta\gamma^{-1}$ or of string type
$\alpha\beta^{-1}\eta\gamma^{-1}\alpha$  which
does not factor through a projective $\Lambda$-module. 

Altogether, it follows that $B$ contains neither $\beta^{-1}\gamma^{-1}$ nor $\eta^{-1}$ as top-socle 
piece, which means that $B$ does not exist and hence Proposition \ref{prop:bands} follows.
\end{proof}


\end{document}